\documentclass[reqno,11pt]{amsart}
\usepackage{amsmath}
\usepackage{amsthm}
\usepackage{amssymb}
\usepackage{epsfig}
\usepackage[all,dvips]{xy}
\usepackage{tikz}
\usepackage{enumerate}

\newcommand{\Z}{\mathbb{Z}}

\newcommand{\D}{\mathcal{D}}

\theoremstyle{plain}
\newtheorem{thm}{Theorem}[section]

\newtheorem{cor}[thm]{Corollary}
\newtheorem{lem}[thm]{Lemma}

\theoremstyle{definition}
\newtheorem{defn}[thm]{Definition}

\newtheorem{example}[thm]{Example}

\theoremstyle{remark}

\title{Derangement Frequency in the Boolean Complex}

\author{K\'{a}ri Ragnarsson}
\address{Department of Mathematical Sciences, DePaul University, Chicago, Illinois}
\email{kragnars@math.depaul.edu}

\author{Bridget Eileen Tenner}
\address{Department of Mathematical Sciences, DePaul University, Chicago, Illinois}
\email{bridget@math.depaul.edu}


\begin{document}

\begin{abstract}

In previous work, we associated to any finite simple graph a particular set of derangements of its vertices.  These derangements are in bijection with the spheres in the wedge sum describing the homotopy type of the boolean complex for this graph.  Here we study the frequency with which a given derangement appears in this set.\\

\end{abstract}

\maketitle

\section{Introduction}

The boolean elements in the Bruhat order on a finitely generated Coxeter system were studied by the second author in \cite{tenner} and characterized by pattern avoidance in some cases. The boolean elements form a simplicial poset and hence form the face poset of a regular face complex, called the \emph{boolean complex} of the Coxeter system. In \cite{ragnarsson-tenner} we studied the homotopy type of this complex, proved that it has the homotopy type of a wedge of spheres of maximal dimension, and gave a recursive formula for computing the number of spheres using edge operations on the underlying, unlabeled Coxeter graph. In \cite{ragnarsson-tenner-2} we gave combinatorial meaning to these spheres. Taking graphs as a starting point, this involved assigning to any ordered graph a set of derangements of its vertex set, and constructing a homology class for each derangement in this set. The homology classes so constructed form a basis for the homology of the boolean complex, and consequently we get a bijection between the derangment sets and the spheres in the wedge sum describing the homotopy type of the boolean complex.

In the present note, we focus on statistical properties of the derangement sets defined in \cite{ragnarsson-tenner-2}. Given a finite ordered set $V$ and a derangement $w$ of $V$, we determine which graphs with vertex set $V$ admit $w$ in their derangement set. This allows us to calculate the frequency of the derangament $w$.

In this note, all graphs are understood to be finite simple graphs. By an \emph{ordered} graph we mean a graph with a total ordering of the vertex set. A \emph{derangement} of a set $V$ is a fixed-point free permutation of the elements of $V$.

\section{The derangement set of an ordered graph}\label{sec:review}
In this section we briefly recall some background material on derangement sets. This material is covered in greater detail in \cite{ragnarsson-tenner-2}, and here we focus only on the key points needed to perform the frequency calculations in Section~\ref{sec:stat}. 

In \cite[Section 3]{ragnarsson-tenner-2}, we describe an algorithm that for an ordered graph $G$ constructs a set $\D(G)$ of derangements of the vertex set of $G$. We refer to $\D(G)$ as the \emph{derangement set} of $G$. This algorithm is recursive, using edge operations to reduce to smaller graphs. We also give an explicit, closed-form description of $\D(G)$ in \cite[Theorem 3.11]{ragnarsson-tenner-2}. This description is used to compute the frequency of derangements in Section~\ref{sec:stat}, and so we take it as the definition of the derangement set in the current paper.

Some notation is necessary before we can state the definition

\begin{defn} \label{def:Canopy}
Let $G$ be an ordered graph and let $w$ be a permutation of its vertex set. For a vertex $t$ in $G$ set
\[ \rho_w(t) = \{t, w(t), \cdots, w^{k-1}(t) \} \, , \]
where $k$ is the smallest positive integer such that $ w^k(t) \leq t$, and set
\[  \lambda_w(t) = w^{-\ell}(t) \, , \]
where $\ell$ is the smallest positive integer such that $w^{-\ell}(t) \leq t$.
\end{defn}

Write $w$ in standard cycle form.  If $t$ is not the smallest element in its cycle, then $\lambda_w(t)$ is the first element appearing to the left of $t$ that is smaller than $t$, and $\rho_w(t)$ is the set of elements obtained by starting at $t$ and moving to the right until reaching an element less than $t$. When $t$ is the smallest element in its cycle, $\rho_w(t)$ is the set of elements in the entire cycle containing $t$, and $\lambda_w(t) = t$.

\begin{defn}
\label{defn:criterion}
Let $G$ be an ordered graph with vertex set $V$. The \emph{derangement set} of $G$, denoted $\D(G)$, is the set consisting of permutations $w$ of $V$ such for every vertex $t$ of $G$ the vertex $\lambda_w(t)$ is adjacent to some vertex in $\rho_w(t)$.
\end{defn}

It is not hard to check that every $w \in \D(G)$ is indeed a derangement, as the name indicates. Note that $\D(G)$ depends on the chosen ordering of the vertex set of $G$.

\begin{example}\label{ex:criterion}
Let $G$ be the 7-vertex graph depicted below.
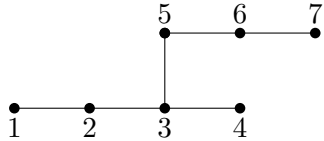
\begin{center}
\begin{tikzpicture}
\draw (1,0) -- (4,0);
\draw (3,0) -- (3,1);
\draw (3,1) -- (5,1);
\foreach \x in {1,2,3,4} {\fill[black] (\x, 0) circle (2pt); \draw (\x,0) node[below] {$\x$};}
\foreach \x in {3,4,5} {\fill[black] (\x, 1) circle (2pt);}
\fill[black] (3,1) circle (2pt);
\draw (3,1) node[above] {$5$};
\draw (4,1) node[above] {$6$};
\draw (5,1) node[above] {$7$};
\end{tikzpicture}
\end{center}
It is easy to see that $(1234)(567) \in \D(G)$.  On the other hand, the derangement $(1234567) \not\in \D(G)$ because $\lambda_{(1234567)}(5) = 4$ and $\rho_{(1234567)}(5) = \{5,6,7\}$, but $4$ is not adjacent to $5$, $6$, or $7$ in the graph.  Similarly, the derangement $(13472)(56)$ is excluded from $\D(G)$ because $\lambda_{(13472)(56)}(3) = 1$ and $\rho_{(13472)(56)}(3) = \{3,4,7\}$, and $1$ is not adjacent to $3$, $4$, or $7$.
\end{example}

\section{Derangement frequency}\label{sec:stat}

We now compute the frequency of a derangement, defined as follows.

\begin{defn}
Given a derangement $w$ of a finite, totally ordered set $V$, let $f(w)$ be the number of simple ordered graphs $G$ with vertex set $V$ such that $w \in \D(G)$, and let 
 $$r(w) = \frac{f(w)}{2^{\binom{|V|}{2}}} \, .$$
The statistic $f(w)$ is the \emph{frequency} of $w$ and $r(w)$ is the \emph{rate} of $w$.
\end{defn}

The frequency of a derangement $w$ of the set $V$ describes how often $w$ appears in the derangement sets for graphs with vertex set $V$, while the rate of $w$ calculates the proportion of all graphs on $V$ with $w$ in their derangement sets.

Given a derangement $w$ of $V$, Definition~\ref{defn:criterion} requires that a graph $G$ with $w \in \D(G)$ must contain at least one edge in the following sets for each vertex $t$ in $G$.

\begin{defn}
Given a vertex $t$ in the finite simple ordered graph $G$ and a $w \in \D(G)$, let $E_{w,t} = \left\{\{\lambda_w(t), s\} : s \in \rho_w(t)\right\}$.
\end{defn}

Note that when $t$ is not minimal in its $w$-cycle, we have $\lambda_w(t) < s$ for all $s \in \rho_w(t)$.  On the other hand, when $t$ is minimal in its $w$-cycle, we have $t = \lambda_w(t) < s$ for all $s \in \rho(w) \setminus \{t\}$. These cases often have to be treated separately, and the ensuing discussion is simplified by the following definition.

\begin{defn}
Give a derangement $w$ of an ordered set $V$, let $U(w)$ be the set of elements that are not minimal in their $w$-cycles.
\end{defn}

Calculating the frequency of a derangement amounts to counting the possible edge sets that contain at least one edge in each $E_{w,t}$.  This is greatly simplified by the following lemma. The proof of the lemma uses the observation following Definition~\ref{def:Canopy}. 

\begin{lem}\label{lem:disjoint edge requirements}
Let $w$ be a derangement of a totally ordered set $V$. 
\begin{enumerate} [(a)]
  \item For distinct $t \in U(w)$, the sets $E_{w,t}$ are disjoint.
  \item If $s$ is minimal in its $w$-cycle then $E_{w,w(s)} \subseteq E_{w,s}$.
\end{enumerate} 
\end{lem}

\begin{proof}
We consider $w$ written in standard cycle form.
\begin{enumerate} [(a)]
  \item Consider $t, t' \in U(w)$.  If $t$ and $t'$ are in difference $w$-cycles, then certainly $E_{w,t} \cap E_{w,t'} = \emptyset$.  Now suppose that $t$ and $t'$ are in the same cycle of $w$, and that $t < t'$.  To have $\lambda_w(t) = \lambda_w(t')$, we must have $t'$ appearing to the left of $t$ when written in standard cycle form.  But then the set of elements $\rho_w(t')$ is a string of consecutive elements beginning with $t'$, extending to the right, and stopping before reaching $t$.  Thus $\rho_w(t') \cap \rho_w(t) = \emptyset$, so again $E_{w,t} \cap E_{w,t'} = \emptyset$.
  \item Let $s$ be minimal in its $w$-cycle.  Then $\lambda_w(s) = s$ and $\rho_w(s)$ is the entire cycle containing $s$.  Certainly, then, $\rho_w(w(s)) \subseteq \rho_w(s)$.  Moreover, $\lambda_w(w(s))$ is the first element less than $w(s)$ which appears to the left of $w(s)$.  The element immediately to the left of $w(s)$ is necessarily $s$, and $s$ is the minimal element in the cycle so we must have that $s < w(s)$.  Therefore $\lambda_w(w(s)) = s = \lambda_w(s)$, so $E_{w,w(s)} \subseteq E_{w,s}$.
\end{enumerate}
\end{proof}

Lemma~\ref{lem:disjoint edge requirements} shows that to construct a graph $G$ with $w \in \D(G)$, one can independently choose edges from the sets $E_{w,t}$ with $t \in U(w)$ and in doing so also end up with edges from the sets $E_{w,s}$ with $s \notin U(w)$. From this it is straightforward to derive the following description of the frequency and rate of a derangement.

\begin{thm}\label{thm:frequency}
Let $w$ be a derangement of a totally ordered set $V$.  Then the frequency of $w$ is
\begin{eqnarray*}
f(w) &=& 2^{\binom{|V|}{2} - \sum_{t \in U(w)} |\rho_w(t)|} \cdot \prod_{t \in U(w)} \left(2^{|\rho_w(t)|} - 1\right)\\
&=& 2^{\binom{|V|}{2}} \cdot \prod_{t \in U(w)} \left(1 - \frac{1}{2^{|\rho_w(t)|}}\right),
\end{eqnarray*}
and the rate of $w$ is
$$r(w) = \prod_{t \in U(w)} \left(1 - \frac{1}{2^{|\rho_w(t)|}}\right).$$
\end{thm}

\begin{proof}
There are $\binom{|V|}{2}$ possible edges in a graph with vertex set $V$.  Only $\sum_{t \in U(w)}|\rho_w(t)|$ of these are elements of
$$\bigcup_{t \in U(w)} E_{w,t}.$$
Thus the remaining edges can be in the graph or not in the graph, yielding the factor 
$$2^{\binom{|V|}{2} - \sum_{t \in U(w)} |\rho_w(t)|}.$$
For each $t \in U(w)$, a nonempty subset of the edges in $E_{w,t}$ must be present in the graph, yielding the factor $2^{|\rho_w(t)|} - 1$.

The rate of $w$ is obtained by dividing the frequency by the total number of graphs with vertex set $V$.
\end{proof}

\begin{example}
Consider $w = (13472)(56)$.  Then
$$f(w) = 2^{14} \left(2^1 - 1 \right)\left(2^3 - 1\right)\left(2^2 -1\right)\left(2^1-1\right)\left(2^1 -1 \right) = 2^{14} \cdot 21$$
and
$$r(w) = \frac{21}{2^7}.$$
This means that there are $2^{14} \cdot 21$ ordered graphs $G$ for which $w \in \D(G)$, and that of all ordered graphs $G$ on $7$ vertices, $21/2^7$ of them have $w \in \D(G)$.
\end{example}

The frequency and rate function are, in a certain sense, increasing with respect to the sizes of the sets $\rho_w(t)$. To make this precise, we define the following objects.

\begin{defn}
Let $\D^k(V)$ be the set of derangements of $V$ with $k$ disjoint cycles. For $w \in \D^k(V)$ let $\theta(w) \in (\Z_+)^{|V|-k}$ be the $(|V|-k)$-tuple consisting of the numbers $\{|\rho_w(t)|: t \in U(w)\}$, written in non-decreasing order. We order the $(|V|-k)$-tuples $\{\theta(w): w \in \D^k(V)\}$ via the cartesian order.
\end{defn}

Using this as a basis for comparing derangements, we obtain the following monotonicity result as a direct consequence of Theorem \ref{thm:frequency}.

\begin{cor}\label{cor:FreqIncr}
Let $V$ be an ordered set and $w_1,w_2 \in \D^k(V)$ for some $k$.  If $\theta(w_1) \leq \theta(w_2) $ then $f(w_1) \leq f(w_2)$ and $r(w_1) \leq r(w_2)$.
\end{cor}

\begin{example}
Let $w = (13472)(56)$, and let $w' = (13427)(56)$.  Then $\theta(w) = (1,1,1,2,3)$ and $\theta(w') = (1,1,1,2,2)$, so $\theta(w) > \theta(w')$.  Also, $f(w) = 172032$ and $f(w') = 147456$, while $r(w) = .08203125$ and $r(w') = .0703125$.
\end{example}

\section{Extremal examples}

We now give examples of the frequency and rate functions for some extremal cases.  In each of these, the vertex set is $V = \{1, \ldots, n\}$.

We begin by examining rare derangements; that is, derangements with minimal frequency.

\begin{example}\label{ex:fixed number of cycles}
For derangements with a fixed number of cycles, $\theta(w)$ is minimized when every vertex that is not minimal in its cycle is followed by a smaller vertex.  That is, suppose that $\rho_w(t) = \{t\}$ for each $t \in U(w)$, meaning that each cycle can be written as a decreasing sequence.  Then the frequency of $w$ is
$$\frac{2^{\binom{n}{2}}}{2^{|U(w)|}}$$
and the rate of $w$ is
$$\frac{1}{2^{|U(w)|}}.$$
If we allow the number of cycles to vary, then this value is minimized if there is only one value of $t$ which is minimal in its cycle, meaning that the standard cycle form of $w$ has exactly one cycle.  In other words, this is minimized for $w = (1n \cdots 432)$:
$$f((1n \cdots 432)) = \frac{2^{\binom{n}{2}}}{2^{n -1}} \text{\ \ \ and\ \ \ } r((1n \cdots 432)) = \frac{1}{2^{n -1}}.$$
On the other hand, the derangements with the highest rate in this situation are those where the standard cycle form of $w$ has as many cycles as possible, namely $\lfloor n/2 \rfloor$ cycles, in which case the rate of $w$ is
\begin{equation}\label{eqn:maximizing}
\frac{1}{2^{\lceil n/2 \rceil}}.
\end{equation}
\end{example}

Complementary to the previous example, let us now consider persistent derangements; that is, derangements with maximal frequency.

\begin{example}\label{ex:fixed cycle lengths}
Suppose that the cycle lengths in the standard cycle form of $w$ are $c_1, c_2, \ldots$.  If, for each non-minimal element in each cycle, $|\rho_w(x)|$ is maximal, meaning that each cycle can be written as an increasing sequence, then the frequency of $w$ is
$$2^{\binom{n}{2}} \cdot \prod_i \prod_{k=1}^{c_i-1} \left(1 - \frac{1}{2^k}\right).$$
and the rate of $w$ is
$$\prod_i \prod_{k=1}^{c_i-1} \left(1 - \frac{1}{2^k}\right).$$
This value is maximized if there is only one cycle in the standard form of $w$, and hence $c_1 = n$.  In other words, this is maximized for $w = (1234\cdots n)$:
$$f((1234\cdots n)) = \prod_{k=1}^{n-1} \left(2^k - 1\right) \text{\ \ \ and\ \ \ } r((1234\cdots n)) = \frac{\prod_{k=1}^{n-1} \left(2^k - 1\right)}{2^{\binom{n}{2}}}.$$
On the other hand, the derangements with the lowest rate in this situation are those where the standard cycle form of $w$ has as many cycles as possible, meaning that all cycles are transpositions except possibly for a single $3$-cycle in the case when $n$ is odd.  If $n$ is even, then the rate of $w$ is
\begin{equation}\label{eqn:minimizing even}
\frac{1}{2^{n/2}},
\end{equation}
whereas if $n$ is odd, then the rate of $w$ is
\begin{equation}\label{eqn:minimizing odd}
\frac{3}{2^{(n+1)/2}}.
\end{equation}
\end{example}

Notice that varying the number of cycles has the opposite effect in Examples~\ref{ex:fixed number of cycles} and~\ref{ex:fixed cycle lengths}.
  Thus we cannot do much to improve Corollary~\ref{cor:FreqIncr}.  Also, observe that expressions \eqref{eqn:maximizing} and \eqref{eqn:minimizing even} agree (when $n$ is even), as they ought to do.  Expressions \eqref{eqn:maximizing} and \eqref{eqn:minimizing odd} disagree (when $n$ is odd), because the $3$-cycle in the former situation can be written in decreasing order, whereas the $3$-cycle in the latter situation can be written in increasing order.  Stated another way, the least frequently occurring derangement is
$$(1n\cdots 432)$$
and the most frequently occurring derangement is
$$(1234\cdots n).$$

\end{document}